\documentclass{amsart}
\usepackage{amssymb,amsmath}

\usepackage{amsfonts}
\usepackage{amsmath,amscd}
\usepackage{amssymb}
\usepackage{amsthm}
\usepackage{newlfont}
\newtheorem{thm}{Theorem}[section]
 
 \newtheorem{lem}[thm]{Lemma}
 
 \theoremstyle{definition}
 
 \theoremstyle{remark}

 \numberwithin{equation}{section}

\newcommand{\ds}{\displaystyle}

\begin{document}
\title[The Non-Abelian Tensor Square and Schur multiplier of Groups of Orders]{The Non-Abelian Tensor Square and Schur multiplier of
Groups of Orders $p^2q$, $pq^2$ and $p^2qr$ }
\author{S. H. Jafari}
\address{Department of
Mathematics, Ferdowsi University of Mashhad, Mashhad, Iran.}
\author{P. Niroomand}
\address{School of Mathematics and Computer Science,
Damghan University of Basic \\Sciences,  Damghan, Iran. }
\author{A. Erfanian}
\address{Department of
Mathematics, Ferdowsi University of Mashhad, Mashhad, Iran.}
\date{}
\maketitle
\begin{abstract}
The aim of this paper is to determine the non-abelian tensor
square and Schur multiplier of groups of square free
order and of groups of orders $p^2q$, $pq^2$
and $p^2qr$, where $p$, $q$ and $r$ are primes and $p<q<r$.

\end{abstract}

\section{Introduction}
The notion of non-abelian tensor product $G\otimes H$ of groups $G$
and $H$ which was introduced by R. Brown and J.-L. Loday \cite{br, br2} is a generalization of the usual tensor product of the
abelianized groups, on the assumption that each of $G$ and $H$ acts
on the other. Specially, for given groups $G$, $H$ each of which
acts on the other
\[ G\times H\longrightarrow G, (g,h)\longmapsto ^hg \ \ ; \ \ H \times G \longrightarrow G, (h,g)\longmapsto ^gh \] in such a way
that for all $g,g_{1} \in G$ and $h,h_{1} \in H$, \[ ^{^{^{g_{1}}}
h}g=~^{g_{1} ^{-1}}(^h (^{g_{1}}g)) \ \ \ \textrm{and} \ \ \
^{^{^{h_{1}}} g}h=~^{h_{1} ^{-1}}(^g(^{
h_{1}}h)),\hspace{2cm}(\star)
 \]
where $G$ and $H$ acts on themselves by conjugation. Then the
non-abelian tensor product $G\otimes H$ is defined to be the group
generated by symbols $g\otimes h$,  $g \in G, h \in H$, subject to
the relations
\[g_1g\otimes h=(^{g_1}g\otimes \hspace{.04cm} ^{g_1}h)(g_1 \otimes h) \ \ \ , \ \ \
g\otimes h_1h=(g\otimes h_1)\ (^{h_1}g\otimes\hspace{.04cm}^{h_1}h )
\] for all $g,g_1 \in G$, $h,h_1 \in H$, where $G$ acts on itself by
conjugation, i.e. $^{g}g_{1}=gg_1g^{-1}$, and similarly for $H$. In
particular, as the conjugation action of a group $G$ on itself
satisfies $(\star)$, the tensor square $G\otimes G$ may always be
defined. \\
\indent Following the publications of Brown and Loday's work, a
number of purely group theoretic papers have appeared on the topic.
Some investigate general structural properties of the tensor square,
while  the others are devoted to explicit descriptions for
particular groups, for instance dihedral, quaternionic, symmetric
and all groups of order at most $30$ in \cite{br3}.

 Later, Hannebauer
\cite{ha} determined the structure of tensor square of the groups
$SL(2,q),$ $PSL(2,q),$ $GL(2,q)$ and $PGL(2,q)$ for all $q \geqslant 5$
and $q \neq 9$. \\Ellis and Leonard \cite{ge} devise a computer algorithm for the
computation of tensor square of finite groups which in its
applications can handle much larger groups than those given in
\cite{br3}. Using the CAYLEY-program, they
compute the tensor square of $B(2,4)$, the 2-generator Burnside
group of exponent $4$, where $|B(2,4)|=2^{12}$. Recently, it is also improved for the groups
$SL(n,q), PSL(n,q), GL(n,q), PGL(n,q)$ for all $n,q \ge 2$ in
\cite{er} and for extra special p-groups
($p\neq 2$) in \cite{ers}. These results
extremely depend on knowing the order of $M(G)$, the Schur
multiplier of a group $G$.

In the present paper, we focus on the non-abelian tensor square of the group of
square free order and groups of orders $p^2q$, $pq^2$ and $p^2qr$,
where $p$, $q$ and $r$ are primes and $p<q<r$. It should be mentioned, generally finding the structure of $M(G)$ and $G\otimes G$ for a
class of groups is not so easy. Although some computations of $G\otimes G$ for metacyclic and polycyclic groups have been done in
\cite{br3, mo} without knowing $M(G)$. So our result included some special cases which have been summarized as follows.\\

{{\bf{Theorem A.}}}~ \textit{Let $G$ be a group of order
$n$, where $n$ is a square free number. Then
\[\ds G\otimes G \cong {\mathbb{Z}}_{n}.\]}
\indent {{\bf{Theorem B.}}}~ \textit{Let $G$ be a group of
order $p^{2}q$, where $p$ and $q$ are prime numbers
and $p<q$. The structure of $G\otimes G$ is one of the following\\
($i$)~ If $G^{ab}={\mathbb{Z}}_{p^{2}}$,  then $G\otimes G \cong {\mathbb{Z}}_{p^{2}q} $.\\
($ii$)~ If $G^{ab}={\mathbb{Z}}_{p}\times {\mathbb{Z}}_{p}$,  then $G\otimes G \cong ({{\mathbb{Z}}_{p}})^{4}\times {\mathbb{Z}}_{q}
$.\\
($iii$)~ If $G^{ab}={\mathbb{Z}}_{3}$,  then $G\otimes G \cong  {\mathbb{Z}}_{3}\times Q_{2} $, in which $Q_{2}$ is the quaternion
group of order 8.}\\
\indent {{\bf{Theorem C.}}}~ \textit{Let $G$ be a group of
order $pq^{2}$, where $p$ and $q$ are prime numbers
and $p<q$. The structure of $G\otimes G$ is one of the following\\
($i$)~ If $G'={\mathbb{Z}}_{q^{2}}$, then $G\otimes G \cong {\mathbb{Z}}_{pq^{2}} $.\\
($ii$)~ If $G'={\mathbb{Z}}_{q}\times {\mathbb{Z}}_{q}$ and $M(G)=0$ or if $G'={\mathbb{Z}}_{q}$,
then $G\otimes G\cong{\mathbb{Z}}_{p}\times ({\mathbb{Z}}_{q})^{2}$.\\
($iii$)~ If $G'={\mathbb{Z}}_{q}\times {\mathbb{Z}}_{q}$ and $M(G)={\mathbb{Z}}_{q}$,
then $G\otimes G\cong{\mathbb{Z}}_{p}\times H$, where $H$ is an extra special $q$-group of order $q^{3}$.}\\
\indent {{\bf{Theorem D.}}}~ \textit{Let $G$ be a group of
order $p^{2}qr$, where $p$, $q$ and $r$ are prime numbers
 $p<q<r$, $pq\ne 6$. The structure of $G\otimes G$ is one of the following\\
($i$)~ If $|G'|=q$, $r$ or $qr~(q\nmid r-1)$,
 then $G\otimes G \cong {\mathbb{Z}}_{p^{2}qr} $ when $G^{ab}$ is cyclic, otherwise
 $G\otimes G \cong ({\mathbb{Z}}_{p})^{4}\times {\mathbb{Z}}_{qr} $.\\
 ($ii$)~ If $|G'|=qr$~$(q\mid r-1)$,
 then $G\otimes G \cong {\mathbb{Z}}_{p^{2}}\times G' $ when  $G^{ab}$ is cyclic,
 otherwise $G\otimes G \cong ({\mathbb{Z}}_{p})^{4}\times G' $.}
\section{Basic Results}
In this section, we recall some definitions and basic results on the
tensor square which are necessary for our main theorems.

\indent Let $G$ be a group and $G\otimes G$ be the tensor square of
$G$. The exterior square $G\wedge G$ is obtained by  imposing the
additional relation $g\otimes g=1$ $(g \in G)$ on $G\otimes G.$
Moreover, we denote by $\nabla(G)$ as the subgroup generated by all
elements $g\otimes g$. The commutator map induces homomorphisms
$\kappa:G\otimes G\rightarrow G$, and $\kappa':G\wedge G \rightarrow
G$ sending $g\otimes h$ and $g\wedge h$ to $[g,h]=ghg^{-1}h^{-1}$,
denoted $\ker\kappa$ by $J_2(G)$.

\indent Results in Brown, Loday \cite{br,br2} give the
following commutative diagram with exact rows
and central extensions as columns
\begin{equation}\begin{CD}
 @.  0 @. 0\\
 @. @VVV   @VVV\\
\Gamma(G^{ab}) @>>>J_2(G)@>>>M(G)@>>>0\\
 @| @VVV @VVV\\
\Gamma(G^{ab}) @>>>G\otimes G @>>>G\wedge G@>>>1 @. \\
 @. @V\kappa VV   @V\kappa'VV\\
 @. G' @= G'\\
 @. @VVV   @VVV\\
 @. 1 @. 1\\
\end{CD}\end{equation}\\
where $\Gamma$ is the Whitehead's quadratic functor (see Whitehead \cite{wh}).\\
\indent Now let us remind the following theorems which play an
important role in
the proof of the main theorems. \\

\begin{thm}\label{a}\cite[Proposition 8]{br3}.  If $G$ is a
group in which $G'$ has a cyclic complement $C$, then $G\otimes G\cong (G\wedge
G)\times G^{ab}$ and hence $|G\otimes G|= |G| |M(G)|.$
\end{thm}

\begin{thm}\label{b}\cite[Theorem A]{ers}.
Let $G$ be a group such that
$G^{ab}=\ds\prod_{i=1}^n\prod_{j=1}^{k_i}\mathbb{Z}_{p_i^{e_{ij}}}$
where $1\leq e_{i1}\leq e_{i2}\leq ... \leq e_{ik_i}$ for all $1\leq
i \leq n$, $k_i\in \mathbb{N}$ and $p_i\neq 2$. Then
\[\ds|G\otimes G|=\prod_{i=1}^np_i^{d_i}|G||M(G)|\] in which $\ds
d_i=\sum_{j=1}^{k_i}(k_i-j)e_{ij}.$
\end{thm}

\section{Proof of Main Theorems }
In this section we prove the main theorem as we mentioned earlier
in section one.

\begin{lem}\label{c} Let $G$ be a finite non-abelian group.\\
($i$)~ If $G$ is a square free order group, then $G'$ is cyclic.\\
($ii$)~ If $G$ is a group of order $p^{2}q$, then $G'=\mathbb{Z}_{q}$ or $G'=\mathbb{Z}_{2}\times \mathbb{Z}_{2}$.\\
($iii$)~ If $G$ is a group of order $pq^{2}$, then $G'=\mathbb{Z}_{q}$ or $G'=\mathbb{Z}_{q^{2}}$ or $G'=\mathbb{Z}_{q}\times
\mathbb{Z}_{q}$.\\
($iv$)~ If $G$ is a group of order $p^{2}qr$ and $pq\neq 6$, then $|G'|=q, r $ or $qr$.\\
($v$)~ If $G$ is a group of order $p^{2}qr$ and $pq=6$, then
$|G'|=3,r,3r,4$ or $4r$. \end{lem}

\begin{proof}
One may use sylow theorems for examples in case ($ii$) to show the
number of $p$-sylow subgroups of $G$ is 1 or $q$ and the number of
$q$-sylow subgroups of $G$ is 1 or $p^{2}$. If $G$ has one $q$-sylow
subgroup $Q$, then $G/Q$ is abelian and so $G'=\mathbb{Z}_{q}$.
Otherwise $p=2$ and $q=3$, hence $|G|=12$ and we should have $G\cong A_{4}$. \\
\indent The proof of other cases is similar and we omit
it.\end{proof}

 \begin{lem}\label{d}Let $G$ be a finite non-abelian group.\\
($i$)~ If $G$ is a square free order group, then $M(G)=0$ .\\
($ii$)~ If $G$ is a group of order $p^{2}q$, then
\[M(G)=\left\{\begin{array}{lcl}
0&& ~~~~if~~~G'=\mathbb{Z}_{q}~and~ G^{ab}=\mathbb{Z}_{p^{2}}\vspace{.3cm}\\
\mathbb{Z}_{p}&&~~~~if~~~G'=\mathbb{Z}_{q}~and~ G^{ab}=\mathbb{Z}_{p}\times \mathbb{Z}_{p}\vspace{.3cm}\\
\mathbb{Z}_{2}&&~~~~if~~~G'=\mathbb{Z}_{2}\times \mathbb{Z}_{2}\\
\end{array}\right.\]

\noindent($iii$)~ If $G$ is a group of order $pq^{2}$, then
\[M(G)=\left\{\begin{array}{lcl}
0&& ~~~~if~~~G'=\mathbb{Z}_{q}\vspace{.3cm}\\
0&& ~~~~if~~~G'=\mathbb{Z}_{q^{2}} \vspace{.3cm}\\
0~or~\mathbb{Z}_{q}&&~~~~if~~~G'=\mathbb{Z}_{q}\times \mathbb{Z}_{q}~\\
\end{array}\right.\]

\noindent($iv$)~ If $G$ is a group of order $p^{2}qr$ and $pq\neq 6$, then
\[M(G)=\left\{\begin{array}{lcl}
0&&~~~~if~~~ G^{ab}~~is~~cyclic\\
\mathbb{Z}_{p}&&~~~~Otherwise\\
\end{array}\right.\]

\noindent($v$)~ If $G$ is a group of order $p^{2}qr$ and $pq=6$, then $M(G)$ is as the same as part $(iv)$. Moreover if
$|G'|=4~or~4r$,
then $M(G)=\mathbb{Z}_{2}$.
\end{lem}

\begin{proof} $(i)$ Since all Sylow subgroups of G are cyclic, so $M(G)=0$.
For $(ii)$, if $G'=\mathbb{Z}_{q}~and~ G^{ab}=\mathbb{Z}_{p^{2}}$,
then again all Sylow subgroups of G are cyclic and therefore
$M(G)=0$. In the case $G'=\mathbb{Z}_{q}~and~
G^{ab}=\mathbb{Z}_{p}\times \mathbb{Z}_{p}$, we can see that $|G'|$
and $|G^{ab}|$ are coprime, so by Schur-Zassenhaus Lemma, $G'$ has a
complement. The result follows from \cite[Corollary 2.2.6]{ka}. \\The proof of the other parts is similar.\end{proof}

\begin{lem} \label{e} Let $G$ be a finite non-abelian group.\\
\noindent($i$)~ If $G$ is a square free order group of order $n$, then $|G\otimes G|=n$ .\\
\noindent($ii$)~ If $G$ is a group of order $p^{2}q$, then
\[|G\otimes G|=\left\{\begin{array}{lcl}
p^{2}q&&~ if~G^{ab}=\mathbb{Z}_{p^{2}}\vspace{.3cm}\\
{p}^{4}q&& ~if~ G^{ab}=\mathbb{Z}_{p}\times \mathbb{Z}_{p} \vspace{.3cm}\\
24~&& ~if~G^{ab}=\mathbb{Z}_{3}\\
\end{array}\right.\]

\noindent($iii$)~ If $G$ is a group of order $pq^{2}$, then
\[|G\otimes G|=\left\{\begin{array}{lcl}
pq^{2}&&~if~ G'=\mathbb{Z}_{q^{2}}\vspace{.3cm}\\
pq^{2}&&~if~G'=\mathbb{Z}_{q}~or ~G'=\mathbb{Z}_{q}\times \mathbb{Z}_{q}~and ~M(G)=0  \vspace{.3cm}\\
pq^{3}&& ~if~G'=\mathbb{Z}_{q}\times \mathbb{Z}_{q}~and~M(G)=\mathbb{Z}_{q} \\
\end{array}\right.\]

\noindent($iv$)~ If $G$ is a group of order $p^{2}qr$ and $pq\neq 6$, then
\[|G\otimes G|=\left\{\begin{array}{lcl}
p^{2}qr&&~~~~if~~~ G^{ab}~~is~~cyclic\\
{p}^{4}qr&& ~~~~Otherwise\\
\end{array}\right.\]

\noindent($v$)~ If $G$ is a group of order $p^{2}qr$ and $pq=6$, then the order of $G\otimes G$ is similar to the part $(iv)$.
 Moreover if $|G'|=4~or~4r$, then $|G\otimes G|=24r$.\end{lem}

\begin{proof} ($i$)~ Since $|G'|$ and $|G^{ab}|$ are coprime,
 $G'$ has a cyclic complement and
$|G\otimes G|=|G|$ by Schur-Zassenhaus Lemma, Theorem \ref{a} and Lemma \ref{d}. \\ ($ii$)~ Suppose that
$G^{ab}=\mathbb{Z}_{p^{2}}$. In this case $|G'|$ and $|G^{ab}|$ are
coprime and $|G\otimes G|=|G|=p^{2}q$ by Lemma \ref{d}.

Now assume that $G^{ab}=\mathbb{Z}_{p}\times \mathbb{Z}_{p}$, if $p\ne 2$, then $|G\otimes G|=p^{4}q$ by Theorem \ref{b} and Lemma
\ref{d}.
Otherwise $e(\nabla(G))$ divides $e(G)=2q$
and $e(\Gamma(G^{ab}))=4$. Hence $e(\nabla(G))=2$ and
$\nabla(G)=\mathbb{Z}_{2}\times \mathbb{Z}_{2}\times \mathbb{Z}_{2}$
and it implies that $|G\otimes G|=2^{4}q$. \\
We note that if  $G^{ab}=\mathbb{Z}_{3}$, then $G=A_{4}$ and this
case is computed in \cite{br3}.

\noindent ($iii$)~ Suppose that $G'$ is cyclic of order $q^{2}$, so we have $|G\otimes G|=pq^{2}$ by Theorem \ref{a} and Lemma
\ref{d}.
The case that $G'={\mathbb{Z}}_{q}\times {\mathbb{Z}}_{q}$ and $M(G)=0$ holds similarly.
\\Now if $|G'|=q$ and $p\ne2$, then $|G\otimes G|=pq^{2}$ by Theorem \ref{b} and Lemma \ref{d}.
 Also, it is easy to see that $|G\otimes G|=2q^{2}$, if $p=2$.

\noindent($iv$)~ Suppose that $G^{ab}$ is cyclic, so $|G\otimes G|=|G|$ by Theorems \ref{a} and Lemma \ref{d}.

 The other case is similar to the case $(ii)$.

\noindent($v$)~ It is straight forward.\end{proof}

 We are ready to the proof of main theorems, notice that if $G'$ is cyclic, then $G\otimes G$ is abelian.

{{\bf{\emph{Proof of Theorem A.}}}} Since $G$ is a group of
 square free order, then $G\otimes G$ is abelian of order $n$ by Lemmas \ref{c} and \ref{e}.\\

{{\bf{\emph{Proof of Theorem B.}}}} ($i$)~ It is clear that
$G\otimes G$ is an abelian group of order $p^{2}q$ by Lemma \ref{e}. On
the other hand $e(G\otimes G)$ divides $|G\otimes G|$ and the
epimorphism $\pi:G\otimes G\longrightarrow G^{ab}\otimes G^{ab}$
implies that $e(G\otimes G)=p^{2}q$. Hence $G\otimes G \cong
{\mathbb{Z}}_{p^{2}q}$ as required.\\ ($ii$)~  It is as same as the case $(i)$.\\ ($iii$)~ If $G'=\mathbb{Z}_{2}\times
\mathbb{Z}_{2}$, then $G=A_{4}$ and one may refer to the Table $1$
given in \cite{br3}.\\

{{\bf{\emph{Proof of Theorem C.}}}}~($i$)~ The exponent of $G\otimes G$ is equal to
$pq^{2}$, so the proof follows by
Lemma \ref{e}. \\ ($ii$)~ Since $\displaystyle\frac{G\otimes
G}{J_{2}(G)}=G'$ is abelian, we have $(G\otimes G)'\leq J_{2}(G)$,
where $|J_{2}(G)|=p$. In addition the epimorphism $\pi$ implies that
$(G\otimes G)'$ is in it's kernel which is of order $q^{2}$, so
$(G\otimes G)'=1$. The
result holds by Lemma \ref{e} and the fact that $e(G\otimes G)=pq$. \\ ($iii$)~  It is clear that
$\nabla(G)$ and the $q$-sylow subgroup of $G\otimes G$, say $H$, are
normal. Thus by Lemma \ref{e}, $$G\otimes G\cong \nabla(G) \times H,$$
in which $\nabla(G)=\mathbb{Z}_{p}$ and $H$ is of order $q^{3}$.

{{\bf{\emph{Proof of Theorem D.}}}}
  ($i$)~ If $G^{ab}$ is cyclic, then the epimorphism $\pi$ implies that $e(G\otimes G)=p^{2}qr$, so the result follows by Lemma
  \ref{e}.

If $G^{ab}$ is not cyclic, the proof is similar.\\ ($ii$)~
 Assuming $G^{ab}=\mathbb{Z}_{p^{2}}$, then $|\nabla(G)|=|J_{2}(G)|=p^{2}$ and it can be easily seen that $Ker\pi$ is
isomorphic to $G'$ and of order $qr$.
 Moreover since $e(G\otimes G)=p^{2}qr$, Lemma \ref{e} implies that $$G\otimes G\cong \nabla(G)\times Ker\pi\cong \nabla(G)\times
 G'\cong {\mathbb{Z}}_{p^{2}}\times G'.$$ If $G^{ab}=\mathbb{Z}_{p}\times \mathbb{Z}_{p}$, then
 $e(G\otimes G)=pqr$, $|J_{2}(G)|=p^{4}$ and $Ker\pi\cong G'$.
 Therefore Lemma \ref{e} deduces that
   \[G\otimes G\cong J_{2}(G)\times Ker\pi\cong (\mathbb{Z}_{p})^{4}\times G'.\]

Finally, we note that in Theorem D, if $pq=6$, i.e.
$|G|=12r$, then $|G'|=3,r,3r,4$ or $4r$. In the cases that $|G'|=3,r$ or $3r$, the structure of $G\otimes G$ is similar to Theorem D.
In other cases one may show that $G\otimes G\cong {\mathbb{Z}}_{3r}\times Q_{2}$, where $Q_{2}$ is the quaternion group of order 8.

\bibliographystyle{alpha}

\begin{thebibliography}{20}

\bibitem{br}  R. Brown, J.-L. Loday,  Excision homotopique en basse
dimension. \textit{C.R. Acad. Sci. Paris S\'{e}r. I Math} \textbf{298} (1984) 353-356.
\bibitem{br2}  R. Brown,  J.-L Loday, Van Kampen theorems for diagrams of spaces, With an appendix by M. Zisman, \textit{Topology}  \textbf{26} (1987) 311-335.
\bibitem{br3}   R. Brown,  D. L. Johnson, E. F. Robertson, Some Computations of Non-Abelian Tensor Products of Groups,
\textit{J. of Algebra.}  \textbf{111} (1987) 177-202.
\bibitem{er} A. Erfanian, R. Rezaei, S. H. Jafari, Computing the nonabelian tensor square of general linear groups, \textit{Italian J. of
    Pure and Applied Math,} \textbf{24} (2008) 203-210.
\bibitem{ers} A. Erfanian,  S. H. Jafari, R. Rezaei, On the order of nonabelian tensor square of finite groups, \textit{Submitted.}

\bibitem{ha}  T. Hannebauer,  On non-abelian tensor squares of linear groups, \textit{Arch. Math,} \textbf{55} (1990) 30-34.

\bibitem{ge}   G. J. Ellis,  F. Leonard, Computing Schur multipliers and tensor products of finite groups, \textit{Proc. Royal Irish Acad,}
    \textbf{95A} (1995) 137-147.
\bibitem{ka} G. Karpilovsky,  \textit{The Schur multiplier,} London Math. Soc. Monographs, New Series no. 2, 1987.

\bibitem{mo}R. F. Mores, Advances in computing the non-abelian tensor square of polycyclic groups, \textit{Irish Math. Soc. Bulletin,} \textbf{56} (2005) 115-123.
\bibitem{wh}  J. H. C. Whitehead, A certain  exact sequence, \textit{Ann. of Math,} \textbf{52} (1950) 51-110.






\end{thebibliography}

\end{document}